\documentclass[10pt]{amsart}
\usepackage{amsfonts,amsmath,amsthm,amssymb,latexsym,mathrsfs,stmaryrd} \usepackage{mathtools}

\usepackage{tikz}\usetikzlibrary{cd,fit,matrix,arrows,decorations.pathmorphing}
\tikzset{commutative diagrams/.cd}

\newcommand{\Z}{{\mathbb Z}}
\newcommand{\Q}{{\mathbb Q}}
\newcommand{\C}{{\mathbb C}}

\newcommand{\R}{{\mathbb R}}

\def\P{{\mathbb P}}

\DeclareMathOperator{\codim}{codim}

\newcommand{\subeq}{\subseteq}

\newcommand{\ttilde}{\widetilde}

\newcommand{\incl}{\hookrightarrow}

\DeclarePairedDelimiter{\abs}{\lvert}{\rvert}

\DeclarePairedDelimiter{\set}{\{}{\}}
\DeclarePairedDelimiter{\pairing}{\langle}{\rangle}
\DeclarePairedDelimiter{\parens}{\lparen}{\rparen}

\DeclarePairedDelimiter\floor{\lfloor}{\rfloor}

\newcommand{\tu}{\textup}

\usepackage{amsfonts,amsmath,amsthm,amssymb,latexsym,mathrsfs,stmaryrd}
\usepackage{hyperref}

\theoremstyle{theorem}
\newtheorem{theorem}{Theorem}[section]
\newtheorem{corollary}[theorem]{Corollary}
\newtheorem{lemma}[theorem]{Lemma}
\newtheorem{proposition}[theorem]{Proposition}

\theoremstyle{definition}
\newtheorem{definition}[theorem]{Definition}
\newtheorem{example}[theorem]{Example}

\newtheorem{notation}[theorem]{Notation}

\newtheorem{conjecture}[theorem]{Conjecture}
\newtheorem{question}[theorem]{Question}

\newtheorem{remark}[theorem]{Remark}
\newtheorem*{remark*}{Remark}

\theoremstyle{remark}

\numberwithin{equation}{section}

\DeclareMathOperator{\Conf}{Conf}
\DeclareMathOperator{\gr}{gr}
\DeclareMathOperator{\Gr}{Gr}

\def \cA{{\mathcal{A}}}
\DeclareMathOperator{\rk}{rk}
\DeclareMathOperator{\Ind}{Ind}
\newcommand{\dee}{\mathrm{d}}

\begin{document}
\title[Configuration spaces on punctured varieties]{Cohomology of configuration spaces on punctured varieties}
\author{Yifeng Huang}
\date{\today}
\address{Dept of Mathematics, 1984 Mathematics Rd, Vancouver, BC V6T 1Z2, Canada}
\email{huangyf@math.ubc.ca}

\begin{abstract}
In the theory of configuration spaces, ``splitting'' usually refers to the phenomenon that the configuration spaces on a manifold and those on its punctured version are closely related cohomologically. We prove a splitting theorem that is equivariant and mixed-Hodge-theoretic; both are new features in such results. As an application, we determine the generating function for the mixed Hodge numbers of the unordered configuration spaces of a multi-punctured elliptic curve. 
\end{abstract}

\maketitle

\section{Introduction}

For $n\in \Z_{\geq 0}$ and $X$ a topological space or a variety (always assumed to be connected smooth complex quasi-projective in this paper), let $\Conf^n(X)$ denote the configuration space of $n$ unlabeled distinct points on $X$. A recurring theme in configuration spaces is the interaction between $H^*(\Conf^*(X);k)$ and $H^*(\Conf^*(X-P);k)$ with $k=\Z$, $\Z/p\Z$, or $\Q$, often referred to as ``splitting''; see \cite{gorjunov1981cohomology,cohen1993mapping,napolitano2003cohomology,kallel2008symmetric}. In rational coefficients, the most general result up to date is by Kallel \cite[Thm.~1.5]{kallel2008symmetric}:
\begin{equation}\label{eq:kallel}
H^i(\Conf^n(X-P);\Q)\cong \bigoplus_{k=0}^\infty H^{i-(2d-1)k}(\Conf^{n-t}(X);\Q),
\end{equation}
if $X$ can be obtained from removing $r\geq 1$ points from a connected closed orientable manifold $M$ of dimension $2d$. We note it is necessary that $r\neq 0$, so in particular, $X$ cannot be compact without boundary (for example, compare \cite[Tables~2~and~3]{napolitano2003cohomology}). 

The existing approaches to splitting rely on topological constructions that do not stay within the category of complex varieties. Napolitano's approach in \cite{napolitano2003cohomology} to a $\dim_\R X=2$, $\Z$-coefficient analogue of \eqref{eq:kallel} uses a spectral sequence associated with a real cell decomposition. Kallel's approach in \cite{kallel2008symmetric} involves truncated symmetric products, which require gluing in the construction. In this paper, we describe a new approach to splitting that is purely algebraic and explicit. Our approach is based on new observations (Lemmas \ref{lem:decomposition} and \ref{lem:compat-diff}) about a suitable modification of the Cohen--Taylor--Totaro spectral sequence \cite{cohentaylor1978,totaro1996configuration}. The use of such a spectral sequence allows to keep track of mixed Hodge theory and $\ell$-adic cohomology due to Deligne \cite{deligne1971hodge2}, but somewhat surprisingly, this spectral sequence has not been used before in the context of splitting. 

For the rest of this paper, $H^i(-):=H^i(-;\Q)$.

\subsection{Main results}
Our main result says that the splitting \eqref{eq:kallel} is mixed-Hodge-theoretic, equivariant, and applicable to more general non-compact varieties. We introduce the following notion, which is key to our results. 

\begin{definition}\label{def:purity}
    ~
    \begin{itemize}
        \item [(i)]
        Given a rational number $\lambda\geq 0$, we say a variety $X$ is \emph{pure of slope} $\lambda$ if the mixed Hodge structure of $H^i(X,\Q)$ is pure of weight $\lambda\cdot i$ for any integer $i\geq 0$, namely, the mixed Hodge number $h^{p,q;i}(X)$ is zero unless $p+q=\lambda\cdot i$. In particular, $H^i(X,\Q)=0$ for all $i$ such that $\lambda\cdot i$ is not an integer. By \cite[\S 7, p.\ 82]{deligne1975poids}, the slope of a smooth variety (if exists) must satisfy $1\leq \lambda\leq 2$.
        \item [(iii)] We say $X$ is a \emph{non-compact pure variety} if $X$ is a (connected smooth complex) variety that is not compact and is pure of some slope $\lambda$. 
        \item [(iv)]
        We say $Y$ is a \emph{possibly punctured non-compact pure variety} if there is $r\geq 0$ and a non-compact pure variety $X$ such that $Y$ is $X$ minus $r$ points.
    \end{itemize}
\end{definition}

We state first the non-equivariant version of our splitting.

\begin{theorem}[Splitting, non-equivariant]\label{thm:A}
    Let $Y$ be a possibly punctured non-compact pure variety of complex dimension $d$. Then for $p,q,i\geq 0$ and $P\in Y$, 
    \begin{equation}\label{eq:splitting_hodge_direct}
    h^{p,q;i}(\Conf^n(Y-P)) = \sum_{k\geq 0} h^{p-kd,q-kd;i-k(2d-1)}(\Conf^{n-k}(Y)),
    \end{equation}
    where the summand is zero if any of $p-kd,q-kd,i-k(2d-1),n-k$ is negative.
\end{theorem}

By fixing $i$ and summing over all $p,q$, Theorem \ref{thm:A} recovers Kallel's splitting when the closed orientable manifold $M$ in the setting of \eqref{eq:kallel} is a smooth compact complex variety. This is because $M$ minus one point is a non-compact and pure of slope 1 by \cite[Thm.~2.10]{dupont2016purity}. 
Example \ref{eg:pure-slope} shows that non-compact pure varieties exist in greater generality. Theorem \ref{thm:A} thus gives some first new examples where \eqref{eq:kallel} holds, such as smooth plane curve complements in $\P^2_{\C}$ (Example \ref{eg:pure-slope}(iii)). 

We now move on to the equivariant version. Below, we gather some standard notation and terminology.

\begin{notation}
    Given a topological space or a variety $X$, let $F(X,n)$ be the configuration space of $n$ \emph{labeled} distinct points on $X$, equipped with an action by the symmetric group $S_n$.

    A mixed Hodge structure is a $\Q$-vector space $M$ equipped with filtrations $F^{\bullet}$ on $M_\C$ and $W_{\bullet}$ on $M$ satisfying certain conditions (see \cite{deligne1971hodge2}). Let $M^{p,q}:=\Gr_F^p \Gr_{p+q}^W M_\C$, a complex vector space called the $(p,q)$ part of $M$. For any complex quasi-projective variety $V$, $H^i(V;\Q)$ is canonically equipped with a mixed Hodge structure. Let $H^{p,q;i}(V):=(H^i(V;\C))^{p,q}$ denote the $(p,q)$ part of the mixed Hodge structure of $H^i(V;\Q)$, so $h^{p,q;i}(V)=\dim_\C H^{p,q;i}(V)$. If a finite group $G$ acts on $V$, then $G$ acts on $H^{p,q;i}(V)$ by pulling back, making $H^{p,q;i}(V)$ a complex $G$-representation. 

    To a mixed Hodge structure $M$, we associate a $\Z^2$-bigraded complex vector space
    \begin{equation}
        \gr M:=\bigoplus_{p,q} M^{p,q}
    \end{equation}
    with grading $(p,q)$ referred to as the Hodge type. If a finite group $G$ acts on $M$ as a mixed Hodge structure, then we consider $\gr M$ as a $\Z^2$-graded complex $G$-representation. For $m\in \Z$, let $M(m)$ denote the Tate twist of $M$. It is a mixed Hodge structure that satisfies $M(m)^{p,q}=M^{p+m,q+m}$.

    For $0\leq k\leq n$, consider the subgroup $S_{n-k}:=S_{n-k}\times 1$ of $S_n$ consisting of permutations that fix the last $k$ letters. For a representation $W$ of $S_{n-k}$, let $\Ind_{S_{n-k}}^{S_n} W$ denote the induced $S_n$-representation.
\end{notation}

The equivariant version of our splitting is as follows.
\begin{theorem}[Splitting, equivariant]\label{thm:C}
    Let $Y$ be a possibly punctured non-compact pure variety of complex dimension $d$. Then for $i\geq 0$ and $P\in Y$, we have an isomorphism of $\Z^2$-graded $S_n$-representations
    \begin{equation}\label{eq:equivariant_mhs}
    \gr H^i(F(Y-P,n))\cong \bigoplus_{k\geq 0} \Ind_{S_{n-k}}^{S_n} \gr H^{i-k(2d-1)}(F(Y,n-k))(-kd).
    \end{equation}
\end{theorem}

Theorem \ref{thm:C} implies Theorem \ref{thm:A} by taking the $S_n$-invariants and extracting the dimension of the $(p,q)$ graded piece. To read off the equivariant non-Hodge-theoretic splitting from Theorem \ref{thm:C}, we simply remove $\gr$ and the Tate twist:
\begin{equation}\label{eq:equivariant}
    H^i(F(Y-P,n))\cong_{S_n} \bigoplus_{k\geq 0} \Ind_{S_{n-k}}^{S_n} H^{i-k(2d-1)}(F(Y,n-k)).
\end{equation}

This puts splitting into the framework of \emph{representation stability} (see \cite{church2012homological,churchfarb2013representation,cef2014representation,cef2015fi}). By \cite[Thm.~4.1.7 and 6.4.3]{cef2015fi}, for any $i\geq 0$ and non-compact variety $Y$, the character of the $S_n$-representation $H^i(F(Y,n))$ for all $n$ is given by a single polynomial $P_{Y,i}\in \Q[x_1,x_2,\dots]$, called the \emph{character polynomial} of $H^i(F(Y,-))$:
\begin{equation}\label{eq:charpoly}
    \chi_{H^i(F(Y,n))}(\sigma)=P_{Y,i}(x_1(\sigma),x_2(\sigma),\dots) \text{ for all $n\geq 0$ and $\sigma\in S_n$},
\end{equation}
where $x_k(\sigma)$ is the number of $k$-cycles in $\sigma$. 

Another way to encode $H^i(F(Y,-))$ is through a Schur-positive symmetric function. Let $\mathcal{P}$ be the set of integer partitions, $\Lambda=\bigoplus_{\lambda\in \mathcal{P}} \Z s_\lambda$ be the ring of symmetric functions, where $s_\lambda$ is the Schur function. By \cite[Thm.~4.1.5 and 6.4.3]{cef2015fi}, the $\text{FI}\sharp$-module $H^i(F(Y,-))$ is of the form
\begin{equation}\label{eq:fisharp}
    H^i(F(Y,-)) = \bigoplus_{\lambda\in \mathcal{P}} c_\lambda M(V_{\lambda})
\end{equation}
for some unique $c_\lambda \in \Z_{\geq 0}$, where $M(-)$ is the functor in \cite[Thm.~4.1.5]{cef2015fi} and $V_\lambda$ is the $S_{\abs{\lambda}}$-representation associated with the partition $\lambda$. We define $C_{Y,i}:=\sum_{\lambda} c_\lambda s_\lambda \in \Lambda$, referred to as the \emph{Frobenius characteristic} of $H^i(F(Y,-))$. 

The data $P_{Y,i}$ and $C_{Y,i}$ are equivalent; the rule to convert from $C_{Y,i}$ to $P_{Y,i}$ can be extracted from \cite[p.~1874]{cef2015fi}. However, $C_{Y,i}$ is usually more succinct. For instance \cite[pp.~1839, 1842]{cef2015fi},
\begin{equation}
    P_{\C,2}=2\binom{x_1}{3}+3\binom{x_1}{4}+\binom{x_1}{2}x_2-\binom{x_2}{2}-x_3-x_4,
\end{equation}
while $C_{\C,2}=s_{31}+s_{21}$. In general these invariants are not known explicitly.

We now describe how these invariants associated with $Y-P$ and $Y$ relate.

\begin{corollary}\label{thm:D}
    Let $Y$ be a possibly punctured non-compact pure variety of complex dimension $d$. Then for $i\geq 0$ and $P\in Y$, we have
    \begin{align}
        P_{Y-P,i}(x_1,x_2,\dots)&=\sum_{k=0}^{\floor{\frac{i}{2d-1}}} k! \binom{x_1}{k} P_{Y,i-k(2d-1)}(x_1-k,x_2,\dots),\label{eq:charpoly}\\
        C_{Y-P,i}&=\sum_{k=0}^{\floor{\frac{i}{2d-1}}} s_1^k C_{Y,i-k(2d-1)}.\label{eq:frobchar}
    \end{align}
\end{corollary}

\subsection{Punctured Riemann surfaces}
As an application, we advance the computation of the mixed Hodge numbers of smooth algebraic curves. For $g,r\geq 0$, let $\Sigma_{g,r}$ denote a genus $g$ smooth projective algebraic curve minus $r$ points. Drummond-Cole and Knudson \cite{drummondcoleknudsen} explicitly computed the rational Betti numbers $h^i(\Conf^n(\Sigma_{g,r}))$ for all $i,n,g,r\geq 0$. The mixed Hodge numbers are determined in the cases $g=0, r\geq 1$ (\cite{lehrer1992ladic,kim1994weights}), $g\geq 0, r=0$ (\cite{pagaria2023extra}), and $g=r=1$ (\cite{cheonghuang2022betti}). We formulate a conjecture below that covers all non-compact cases. For a variety $X$, consider the generating function
\begin{equation}\label{eq:hx-def}
    h_X(x,y,u;t):=\sum_{p,q,i,n\geq 0} h^{p,q;i}(\Conf^n(X))x^p y^q (-u)^i t^n\in \Z[x,y,u][[t]].
\end{equation}
A priori, we know $h_X(1,1,u;t)=\sum_{i,n\geq 0} h^i(\Conf^n(X))(-u)^i t^n$ records all Betti numbers, and that $h_X(x,y,1;t)$ is rational and has a Macdonald-type factorization formula. The latter is a consequence of the Macdonald formula \cite[p.~116]{cheah1994cohomology}, the Poincar\'e duality, and a motivic formula for $\Conf^n(X)$ due to Vakil and Wood \cite[Prop.~5.9]{vakilwood2015discriminants}; for a detailed explanation, see \cite[\S 2]{cheonghuang2022betti}. 

\begin{conjecture}\label{conj:generating}
    Let $g\geq 0, r\geq 1$. Then 
    \begin{equation}\label{eq:conj:generating}
        h_{\Sigma_{g,r}}(x,y,u;t)=\frac{1}{(1+xyut)^{r-1}} \frac{\Phi_g\set{(1-xyt^2)(1-xt)^g(1-yt)^g}}{(1-t)(1-x^2 yu^2t^2)^g(1-xy^2u^2t^2)^g},
        \end{equation}
        where $\Phi_g$ is a $\Z[x,y]$-linear map defined by 
        \begin{equation}
        \Phi_g(t^j)=\begin{cases}
        u^j t^j,& 0\leq j\leq g;\\
        u^{j-1}t^j,& g+2\leq j\leq 2g+2.
        \end{cases}
        \end{equation}
\end{conjecture}

Conjecture \ref{conj:generating} is consistent with the known cases for $g=0, r\geq 1$ and $g=r=1$. When $u=1$, the conjecture recovers the aforementioned Macdonald-type formula for $h_{\Sigma_{g,r}}(x,y,1;t)$; note that $\Phi_g$ is the identity map when $u=1$, so the factorization survives. When $x=y=1$, the conjecture recovers \cite[Prop.~3.5]{drummondcoleknudsen}. 

The observation that the data in \cite[Prop.~3.5]{drummondcoleknudsen} can be organized by a piecewise shift of a highly factorizable rational generating function is new. We use an example to demonstrate how the combination of $\Phi_g$ and the sign convention in \eqref{eq:hx-def} ``explains'' a complicated numerator. One could extract from \cite[Prop.~3.5]{drummondcoleknudsen}:
\begin{equation}\label{eq:sigma32}
    \sum_{i,n\geq 0} h^i(\Conf^n(\Sigma_{3,1}))u^i t^n =  \frac{1+6ut+14u^2 t^2+14u^3t^3+14u^4t^5+14u^5t^6+6u^6t^7+u^7t^8}{(1-t)(1-u^2 t^2)^6}.
\end{equation}
A numerical check reveals that the numerator cannot be factorized. To understand the numerator further, one might attempt to shift the exponents of $u$ to align with those of $t$, getting
\begin{equation}
    1+6ut+14u^2 t^2+14u^3t^3+14u^5t^5+14u^6t^6+6u^7t^7+u^8t^8,
\end{equation}
but it is still irreducible in $\Z[u,t]$. However, if we negate the signs of the top half of the coefficients, we get an excellent factorization:
\begin{equation}
    1+6ut+14u^2 t^2+14u^3t^3-14u^5t^5-14u^6t^6-6u^7t^7-u^8t^8=(1-ut)(1+ut)^7.
\end{equation} 
The above process can be recovered from taking $g=3, r=1$ and $x=y=1, u\mapsto -u$ in \eqref{eq:conj:generating}. The above demonstration also indicates that some sign manipulation is inevitable even though we chose to avoid it in \eqref{eq:sigma32}.

Starting with the known case $g=r=1$ in \cite{cheonghuang2022betti} and applying Theorem \ref{thm:A} repetitively, we obtain:
\begin{corollary}\label{cor:generating}
    Conjecture \ref{conj:generating} holds for $g=1,r\geq 1$. Moreover for a fixed $g\geq 0$, Conjecture \ref{conj:generating} holds for all $r\geq 1$ if and only if it holds for $r=1$.
\end{corollary}

Note that one could not use $\Sigma_{g,0}$ to study Conjecture \ref{conj:generating} since Theorem \ref{thm:A} does not connect $\Sigma_{g,0}$ and $\Sigma_{g,1}$: in Theorem \ref{thm:A}, both varieties in comparison must be non-compact. The attempt to attack Conjecture \ref{conj:generating} using $\Sigma_{g,1}$ is an ongoing project with Eric Ramos.

\subsection{Remarks about our results}
\subsubsection{}
To fuel examples of our Theorems \ref{thm:A} and \ref{thm:C}, we provide some non-compact pure varieties below.

\begin{example}\label{eg:pure-slope}
    ~
    \begin{itemize}
        \item[(i)] Let $M$ be a smooth projective variety, $P\in M$, and $X=M-P$. Then $H^i(X)$ is pure of weight $i$ \cite[Thm.~2.10]{dupont2016purity}, so $X$ is pure of slope $1$.\label{eg:punc}
        \item[(ii)] Let $X$ be the complement of a hyperplane arrangement in $\C^d$ or a toric arrangement in $(\C^*)^d$. Then $H^i(X)$ is pure of weight $2i$ by \cite[Thm.~3.7, 3.8]{dupont2016purity}, so $X$ is pure of slope $2$. 
        \item[(iii)] Let $C$ be a smooth plane curve of genus $g$ in $\P^2$, and let $X=\P^2-C$. From the Gysin sequence (see the exposition \cite[\S 6]{durfee1983naive}), we obtain
        \begin{equation}
        H^i(X)=\left\{\hspace{-1ex}\begin{array}{lll}
        \Q, &\text{ pure of weight 0}, &i=0\\
        \Q^{2g}, &\text{ pure of weight 3}, &i=2\\
        0, & &i\neq 0,2.
        \end{array}\right.
        \end{equation}
        Hence $X$ is pure of slope $3/2$.
        \item[(iv)] Suppose $X$ is smooth and pure of slope $\lambda$, and $G$ is a finite group that acts on $X$ freely such that the scheme-theoretic quotient $X/G$ is also smooth. Then from $H^i(X/G)=H^i(X)^G$ we conclude that $X/G$ is also pure of slope $\lambda$. As a sample application, the generalized configuration space $F(\C,m)/G$ for a subgroup $G$ of $S_m$ is pure of slope $2$.
    \end{itemize}
\end{example}

\subsubsection{}
Nice behaviors of $h_X(x,y,u;t)$ often reflect mixed Hodge purity. For example, $H^i(\Conf^n(\Sigma_{g,r}))$ is pure of weight $2i$ if $g=0, r\geq 1$ (\cite{lehrer1992ladic,kim1994weights}) and pure of weight $\floor{3i/2}$ if $g=r=1$ (\cite{cheonghuang2022betti}), and they explain the corresponding cases of Conjecture \ref{conj:generating}. Note in addition that Church, Ellenberg, and Farb \cite{cef2014representation} beautifully use a purity argument to connect the cohomology of $\Conf^n(\R^2)=\Conf^n(\Sigma_{0,1})$ to Hasse--Weil zeta functions and polynomial statistics over finite fields. However, the other verified cases of Conjecture \ref{conj:generating} (as in Corollary \ref{cor:generating}) cannot be explained by such a purity argument. For example, if $X=\Sigma_{1,2}$, then even $\Conf^1(X)=X$ is not pure: $H^1(X)$ has a mixture of weight $1$ and $2$. 

The use of a shifted highly factorizable rational function in Conjecture \ref{conj:generating} is also forshadowed in our past work with Cheong \cite{cheonghuang2022betti} on the $g=r=1$ case. There, an equivalent but differently-formulated shift pattern is explained by purity with a quasi-linear (i.e., linear + periodic) weight function $w(i)=\floor{3i/2}$. However, \eqref{eq:conj:generating} (if true) would imply that for any $g\geq 2$, $H^i(\Conf^n(\Sigma_{g,1}))$ is not pure, so it is not yet clear what structure might lead to the shift. 

\subsubsection{}Note that Conjecture \ref{conj:generating} and Corollary \ref{cor:generating} advance a more general question.

\begin{question}\label{question}~
    \begin{itemize}
        \item [\tu{(i)}] Is $h_X(u;t):=\sum_{i,n\geq 0} h^i(\Conf^n(X))(-u)^i t^n$ rational for an even-dimensional orientable manifold $X$? 
        \item [\tu{(ii)}] Is $h_X(x,y,u;t)$ rational for a smooth complex variety $X$?
    \end{itemize}
\end{question}

At least for non-compact varieties $X$, representation stability and the theory of FI$\sharp$-modules says that for each fixed $i$, the behavior of $h^i(\Conf^n(X))$ (and even the finer datum $H^i(F(X,n))$ as an $S_n$-representation) for \emph{all} $n\geq 0$ can be explained by a single invariant, either $P_{X,i}(x_1,x_2,\dots)$ or equivalently $C_{X,i}$; see the discussion before Corollary \ref{thm:D}. However, Question~\ref{question} takes account of all $i$ at the same time, which suggests the possible need to understand the graded FI$\sharp$-\emph{algebra} $H^*(F(X,-))$ along the lines of \cite[Example~5.1.3]{cef2015fi}.

\subsection{Methods and organizations}
Our approach uses a slight generalization of the Cohen--Taylor--Totaro spectral sequence, a tool to compute the cohomology of configuration spaces. But variants of this spectral sequence exist in the context of complements of arrangements or stratified spaces; see for example \cite{looijenga1993cohomology,bibby2016cohomology,dupont2016purity,tosteson2016lattice,petersen2017spectral}. An important point of this paper is that if, instead of the braid arrangement used in the original setting, we consider a slightly different arrangement (\S \ref{sec:spectral-new}), then we will be able to study $H^*(F(Y,n))$ and $H^*(F(Y-P,n))$ on an equal footing, and prove degeneration using a weight argument of Totaro \cite{totaro1996configuration}.

To obtain our key Lemmas \ref{lem:decomposition} and \ref{lem:compat-diff}, we use a presentation of a relevant differential graded algebra along the lines of Totaro. The surprising feature is that our key isomorphism $\Phi$ arises quite artificially from the explicit presentation, but the mere fact that a non-compact variety has vanishing top cohomology implies that $\Phi$ is compatible with the differential. 

The paper is organized as follows. In Section \ref{sec:prelim} we describe the relevant generalization of Totaro's spectral sequence (Proposition~\ref{prop:new}), and prove a degeneration statement (Lemma \ref{lem:degen}). In Section \ref{sec:isom} we explicitly construct the key isomorphism $\Phi$ and prove the key Lemmas \ref{lem:decomposition} and \ref{lem:compat-diff}. In Section \ref{sec:proof} we prove our results as consequences of these key lemmas.

\subsection*{Acknowledgements}
We thank Gilyoung Cheong, Jeffrey Lagarias, Mircea Musta\c{t}\u{a}, Roberto Pagaria, Dan Petersen, Andrew Snowden, John Stembridge, Burt Totaro, and Craig Westerland for helpful conversations and useful comments. We thank Christin Bibby and Philip Tosteson for informing the author about numerous crucial results and references. This work was done with the support of National Science Foundation grants DMS-1840234 and DMS-1701576.

\section*{Notation} For $n\geq 0$, let $[n]:=\set{1,2,\dots,n}$. For $i\in [n]$, $[n]-i:=\set{j\in [n]:j\neq i}$. In this paper, a variety always means a connected smooth complex quasi-projective variety. If $X$ is a variety and $r\geq 0$, then $X_r$ denotes $X-P^1-\dots-P^r$, the complement of $r$ distinct points $P^1,\dots,P^r$ on $X$.

\section{Preliminaries}\label{sec:prelim}
\subsection{Arrangements and Orlik--Solomon algebra}
We follow \cite{orlikterao1992}. A \emph{hyperplane arrangement} is a finite set $\cA$ of complex hyperplanes in $\C^n$. A \emph{stratum} $F$ of $\cA$ is the intersection of zero or more hyperplanes in $\cA$ (so $\C^n$ is also a stratum). The \emph{lattice} $L(\cA)$ of $\cA$ is the partially ordered set (poset) of strata of $\cA$, ordered by inclusion. It has a top element $\hat 1:=\C^n$. The \emph{rank} of an element $F\in L(\cA)$ is the length of any maximal chain from $F$ to $\hat 1$, or equivalently the complex codimension of the corresponding stratum in $\C^n$. A subset $S$ of $\cA$ is called (i) \emph{independent}, if $\bigcap S:=\bigcap_{Y\in S} Y\neq \varnothing$ and $\rk(\bigcap S)= \lvert S\rvert$; (ii) \emph{dependent}, if $\bigcap S\neq \varnothing$ and $\rk(\bigcap S) < \lvert S\rvert$; (iii) \emph{vanishing}, if $\bigcap S=\varnothing$. Note that these notions depend on $\cA$ through $L(\cA)$. 

A \emph{differential graded algebra} is a graded-commutative algebra equipped with a graded derivation. The (rational) \emph{Orlik--Solomon algebra} $A(\cA)$ is a differential graded algebra over $\Q$ given by degree-one generators $g_Y, Y\in \cA$ with relations
\begin{align}
    g_{Y_1}\cdots g_{Y_l}=0,& \text{ if $\set{Y_1,\dots,Y_l}$ is vanishing,} \label{eq:os-rel1}
    \\ \sum_{i=1}^l (-1)^{i-1} g_{Y_{1}} \dots \widehat{g}_{Y_{i}} \cdots g_{Y_{l}} = 0, &\text{ if $\set{Y_1,\dots,Y_l}$ is dependent,} \label{eq:os-rel2}
\end{align}
and derivation given by $\partial g_Y=1$. We refer to the relations \eqref{eq:os-rel1}--\eqref{eq:os-rel2} as the \emph{Orlik--Solomon relations}. Again, they depend only on $L(\cA)$.

To ease the notation, for an ordered subset $S=\set{Y_1,\dots,Y_l}$ of $\cA$, denote $g_S:=g_{Y_1}\cdots g_{Y_l}$. If $S$ is unordered, then $g_S$ is only well-defined up to a sign. The Orlik--Solomon relation can be written compactly as
\begin{equation}\label{eq:os-rel}
    g_S=0\text{ ($S$ vanishing), $\partial g_S = 0$ ($S$ dependent).}
\end{equation}

\subsection{Hyperplane-like arrangements}
A \emph{hyperplane-like arrangement} of codimension $d$ in a smooth variety $V$ is a collection $\cA$ of smooth $d$-codimensional closed subvarieties of $V$ such that
\begin{itemize}
    \item For any $S\subeq \cA$, the intersection $\bigcap_{Y\in S} Y$ (if nonempty) is smooth and connected. Call such an intersection a \emph{stratum} of $\cA$.
    \item Every stratum has codimension a multiple of $d$.
\end{itemize}

Again, let $L(\cA)$ be the poset of strata of $\cA$, called the \emph{lattice} of $\cA$. It is a top element $\hat 1=V$. Any maximal chain from $F\in L(\cA)$ to $\hat 1$ has the same length, called the \emph{rank} of $F$, and we have $\codim(F)=d\cdot \rk(F)$. One can similarly define the notion of independence, dependence, and vanishing of $S\subeq \cA$, the Orlik--Solomon relations on $\cA$, and the Orlik--Solomon algebra $A(\cA)$, all of which depend only on the poset $L(\cA)$.

Finally, let $M(\cA):=V-\bigcup_{Y\in \cA} Y$ denote the \emph{complement} of $\cA$.

\subsection{A spectral sequence}
Our main ingredient is a spectral sequence. Recall that for a smooth closed subvariety $Y$ of a smooth variety $V$ with $\codim_\C Y=d$, its \emph{class} $[Y]\in H^{2d}(X)$ is the image of $1\in H^0(X)$ under the Gysin map $H^0(X)\to H^{2d}(X)$.

\begin{theorem}[{Totaro \cite{totaro1996configuration}}]\label{thm:totaro}
    Let $\cA$ be a hyperplane-like arrangement of codimension $d$ on a variety $V$, and assume that for every stratum $F\in L(\cA)$, the pullback $H^*(V)\to H^*(F)$ is surjective. Then there is a spectral sequence
    \begin{equation}\label{eq:ss}
        E_1^{i,j}(\cA)\implies H^{i-j}(M(\cA);\Q),
    \end{equation}
    where $E_1(\cA)=\bigoplus_{i,j} E_1^{i,j}(\cA)$ is a bigraded algebra we now describe in terms of generators and relations. It is a quotient of $H^*(V)[g_Y: Y\in \cA]$, where $H^i(V)$ has bidegree $(i,0)$ and $g_Y$ has bidegree $(2d,1)$. The relations are the Orlik--Solomon relations \eqref{eq:os-rel}, and
    \begin{equation}\label{eq:rel-totaro}
        g_S \alpha = 0 \text{ if $S\subeq \cA$, $\alpha\in H^*(V)$, $\alpha|_{\cap S}=0\in H^*(\cap S)$.}
    \end{equation}

    The differential $d_1:E_1^{i,j}\to E_1^{i,j-1}$ makes $E_1^{i,j}$ a differential bigraded algebra and is given by $d_1 g_Y=[Y]\in H^{2d}(V)\subeq E_1^{2d,0}(\cA)$ and $d_1 \alpha = 0$ for $\alpha\in H^*(V)$.

    The spectral sequence is functorial and is compatible with mixed Hodge theory, where $g_Y$ is assigned the Hodge type $(d,d)$. 
\end{theorem}

\begin{remark}\label{rmk:leray}
    The spectral sequence \eqref{eq:ss} is the Leray spectral sequence of the arrangement $\cA$, but with a different bidegree convention. More precisely, our $E_1^{i,j}$ is suppported on $j\geq 0, i\geq 2dj$. The $k$-th differential is $d_k: E_k^{i,j}\to E_k^{i+1-k,j-k}$. On the other hand, the Leray spectral sequence
    \begin{equation}
        \mathbf{E}_2^{p,q}(\cA) \implies H^{p+q}(M(\cA);\Q)
    \end{equation}
    is supported on $p\geq 0$, $q\in (2d-1)\Z_{\geq 0}$. The $l$-th differential is $\mathbf{d}_l: \mathbf{E}_l^{p,q}\to \mathbf{E}_l^{p+l,q+1-l}$, so $\mathbf{d}_l$ vanishes unless $l=1+k(2d-1)$ for some $k\in \Z_{\geq 1}$; see \cite{totaro1996configuration}. Let $l=1+k(2d-1)$, $p=i-2dj$, and $q=(2d-1)j$, then these spectral sequences correspond by $E_k^{i,j}=\mathbf{E}_{l}^{p,q}$. Note that the bidegree of $\mathbf{d}_l$ checks up: $\mathbf{d}_l$ maps $\mathbf{E}_l^{p,q}$ into $\mathbf{E}_l^{p+l,q+1-l}=\mathbf{E}_l^{(i+1-k)-2d(j-k),(2d-1)(j-k)}=E_k^{i+1-k,j-k}$. The cohomological degree also checks up: $i-j=p+q$. 

    The difference in the bidegree conventions is not essential and is for the convenience that only the ``important'' pages of the Leray spectral sequences are recorded. 
\end{remark}

\begin{remark}
    Totaro's original statement was for the braid arrangement on $V=X^n$ with $\dim_\C X=d$, but the same proof works for any hyperplane-like arrangement; see \cite[p.~1062, Remark]{totaro1996configuration}. If $d=1$, analogous descriptions of \eqref{eq:ss} appear in Bibby \cite{bibby2016cohomology} and Dupont \cite[Thm.~3.1]{dupont2016purity}. Our spectral sequence can also be obtained \emph{without change of bidegree convention} as a special case of Tosteson's lattice spectral sequence \cite[Thm.~1.8]{tosteson2016lattice}, which reads
    \begin{equation}
    E_1^{i,j}(\cA) = \bigoplus_{F\in L(\cA)} \ttilde H_{j-2}((F,\hat 1); H^i(V, V-F))
    \end{equation}
    where $\ttilde H_{j-2}((F,\hat 1))$ denotes a reduced homology of the order complex of a poset interval (see \cite{wachs2007poset}). For hyperplane-like arrangements, we have $\ttilde H_{j-2}((F,\hat 1))=A_F(\cA):=\mathrm{span}\{g_S: \cap S=F\}\subeq A(\cA)$ (see \cite[Theorem 4.1]{folkman1966homology} and \cite[\S 4.5]{orlikterao1992}). At the same time, we have $H^i(V,V-F)\cong H^{i-2\codim_\C F}(F)$ from Thom isomorphism. Hence $E_1^{i,j}(\cA) = \bigoplus_{\rk F=j} H^{i-2dj}(F)\otimes A_F(\cA)$, or $E_1(\cA)=\bigoplus_{F} H^{*-2d\rk F}(F)\otimes A_F(\cA)$. Since $H^*(V)\to H^*(F)$ is surjective, we may rewrite $H^{*-2d\rk F}(F)$ as a quotient of $H^{*-2d\rk F}(V)$ by an ideal, which eventually leads to our presentation \eqref{eq:ss} using Bibby's argument. The mixed-Hodge compatibility is addressed by Petersen \cite{petersen2017spectral} via a Poincar\'e dual.
\end{remark}

\subsection{Degeneration}
We say a spectral sequence degenerates at $E_2$ if the differentials $d_k$ vanishes for all $k\geq 2$, which ensures $E_\infty=E_2$. We will see this is why we need a purity condition (see Definition~\ref{def:purity}). The lemma below extends the argument of Totaro \cite[\S 4]{totaro1996configuration} for $\lambda=1$ and Dupont \cite[Thm.~3.3, Proof]{dupont2016purity} for $\lambda=2$. 

\begin{definition}
    We say a rational number $\lambda$ is \emph{good} with respect to $d$ if
    \begin{equation}\label{eq:good}
        \frac{\lambda}{2d+\lambda-2d\lambda} \notin \set{2,3,\dots}.
    \end{equation} 
\end{definition}

\begin{lemma}\label{lem:degen}
    Suppose $V$ is pure of good slope $\lambda$ with respect to $d$, and $\cA$ is a hyperplane-like arrangement of codimension $d$ on $V$ satisfying the assumption of Theorem \ref{thm:totaro}. Then the spectral sequence \eqref{eq:ss} degenerates at $E_2$.
\end{lemma}
\begin{proof}
    Recall from Theorem \ref{thm:totaro} that $H^i(V)$ has bidegree $(i,0)$ and weight $\lambda i$, and $g_Y$ has bidegree $(2d,1)$ and weight $2d$. As a result, for $j\geq 0$ and $i\geq 2dj$, by $(i,j)=(i-2dj,0)+j(2d,1)$, $E_1^{i,j}$ is pure of weight $\lambda(i-2dj)+j(2d)=\lambda i - (\lambda-1) 2dj$. For $k\geq 2$, since $E_k^{i,j}$ is a subquotient of $E_1^{i,j}$, it is pure of the same weight.

    Recall from Remark~\ref{rmk:leray} that the $k$-th differential is $d_k: E_k^{i,j}\to E_k^{i+1-k,j-k}$. Fix $k\geq 2$. If we can show that $E_k^{i,j}$ and $E_k^{i+1-k,j-k}$ has different weights, then we must have $d_k=0$, as required. 

    Suppose this is not the case, then we have an equality of weights
    \begin{equation}
        \lambda i - (\lambda-1) 2dj = \lambda (i+1-k) - (\lambda-1) 2d(j-k),
    \end{equation}
    so that 
    \begin{equation}
        \lambda (1-k) + (\lambda-1)2dk=0,\; k=\frac{\lambda}{2d+\lambda-2d\lambda}.
    \end{equation}
    By the definition of good slope, this contradicts with the fact $k\in \set{2,3,\dots}$.
\end{proof}

Our need of the technical notion of good slope is only temporary:

\begin{lemma}\label{lem:good}
    If $X$ is a smooth variety of dimension $d$ that is pure of some slope, then for any $n$, $V:=X^n$ is pure of good slope with respect to $d$.
\end{lemma}
\begin{proof}
    If $X$ is pure of slope $1$, then so is $V$ by K\"unneth's formula. By \eqref{eq:good}, $1$ is always a good slope.

    From now on assume otherwise, so $X$ is non-compact, $H^{\geq 1}(X)\neq 0$ and $X$ is pure of a unique slope $\lambda>1$. (Note that if $H^{\geq 1}(X)=0$ then $X$ is pure of any slope.) Then by K\"unneth $V$ is pure of slope $\lambda$. It remains to show $\lambda$ is good.
    
    Let $i>0$ be such that $H^i(X)\neq 0$. Then $\lambda\cdot i\in \Z$. Since $X$ is non-compact, $H^{2d}(X)=0$, so $i\leq 2d-1$. It follows that
    \begin{equation}
        \lambda \geq 1+\frac{1}{i}\geq \frac{2d}{2d-1}.
    \end{equation}
    But from \eqref{eq:good}, all bad slopes with respect to $d$ are of the form $\dfrac{2dk}{(2d-1)k+1}$, where $k\in \Z_{\geq 2}$. Since $\dfrac{2dk}{(2d-1)k+1} < \dfrac{2d}{2d-1}$, the slope $\lambda$ is good.
\end{proof}

\subsection{Braid-puncture arrangements}\label{sec:spectral-new}
We will consider a slight generalization of the braid arrangement. Consider a connected smooth complex variety $X$ of dimension $d$ and distinct points $P^1,\dots,P^r \;(r\geq 0)$ of $X$. Fix $n\geq 0$, and consider the arrangement $\cA$ of $V:=X^n$ consisting of the following $d$-codimensional closed subvarieties:
\begin{align}
\Delta_{ij}&:=\{(x_1,\dots,x_n)\in X^n: x_i=x_j\},\\
\Delta_i^s&:=\{(x_1,\dots,x_n)\in X^n: x_i=P^s\}
\end{align}
for $\set{i,j}\subeq [n]$ and $1\leq s\leq r$. When $r=0$, $\cA$ is just the braid arrangement. In general, we call $\cA$ the \emph{braid-puncture arrangement} of $X$ with punctures $P^1,\dots,P^r$. Obviously, the complement $M(\cA)$ is the ordered configuration space $F(X_r,n)$, where $X_r=X-P^1-\dots-P^r$.

Clearly $\cA$ is a hyperplane-like arrangement and $L(\cA)$ is isomorphic to the poset of strata of a hyperplane arrangement of $\C^n$ consisting of big diagonals and the ``horizontal or vertical hyperplanes'' $x_i=s$. Moreover, it is clear from a description of strata of $\cA$ that $H^*(X^n)\to H^*(F)$ is surjective for $F\in L(\cA)$ (see Appendix). Recall that the braid arrangement has no vanishing subset and a typical example of dependent subset is, say, $\set{\Delta_{12},\Delta_{23},\Delta_{31}}$, which comes from the fact that $x_1=x_2$ and $x_2=x_3$ imply $x_3=x_1$. For braid-puncture arrangements, however, we have two more sources of dependent and vanishing subsets:
\begin{itemize}
    \item $\set{\Delta_i^s, \Delta_i^t}$ is vanishing if $s\neq t$, since we can't have both $x_i=P^s$ and $x_i=P^t$.
    \item $\set{\Delta_{ij},\Delta_i^s, \Delta_j^s}$ is dependent, since $x_i=P^s$ and $x_i=x_j$ imply $x_j=P^s$.
\end{itemize}

Recall K\"unneth's formula $H^*(X^n)=H^*(X)^{\otimes n}$. More explicitly, let $p_i:X^n\to X$ be the projection onto the $i$-th coordinate, and use the notation below for the rest of the paper:
\begin{equation}
    \alpha_i:=p_i^*(\alpha)\in H^*(X^n) \text{ for }\alpha\in H^*(X).
\end{equation}
Then $H^*(X^n)$ is generated as an algebra by $\alpha_i, \alpha\in H^*(X^n), i\in [n]$. 

We note two more relations in $E_1(\cA)$ from \eqref{eq:rel-totaro}:
\begin{align}
    g_{\Delta_{ij}}(\alpha_i-\alpha_j)&=0,\\
    g_{\Delta_i^s} \alpha_i&=0\text{ for }\alpha\in H^{\geq 1}(X):=\bigoplus_{p\geq 1} H^p(X).
\end{align}

Indeed, we have $(\alpha_i-\alpha_j)|_{\Delta_{ij}}=0$ since $p_i=p_j$ on $\Delta_{ij}$, and we have $\alpha_i|_{\Delta_i^s}=0$ for $\alpha\in H^{\geq 1}(X)$ since $p_i$ is a constant map on $\Delta_i^s$.

It will turn out that the relations from the sources above generate all relations we need (see Appendix). 

Finally, to describe the differential $d_1$, we need to understand $[Y]$ for $Y\in \cA$. For $i\neq j$, let $p_{ij}: X^n\to X^2$ be the projection $(x_1,\dots,x_n)\mapsto (x_i,x_j)$. Let $\Delta$ be the diagonal of $X^2$. We have
\begin{equation}\label{eq:fund-class}
    [\Delta_{ij}]=p_{ij}^*[\Delta],\; [\Delta_i^s]=p_i^*[P^s].
\end{equation}
If furthermore $X$ is non-compact, then $H^{2d}(X)=0$, so $[P^s]=0$ and thus $[\Delta_i^s]=0$. 

To summarize, let $(E_1(X_r,n), \dee)$ denote the differential bigraded algebra $(E_1(\cA), d_1)$ in Theorem~\ref{thm:totaro} for the braid-puncture arrangement $\cA$ above, and denote $g_{ij}:=g_{\Delta_{ij}}$ and $g_i^s:=g_{\Delta_i^s}$. Then we get the following description.

\begin{proposition}\label{prop:new}
The graded commutative algebra $E_1(X_r,n)$ is given by 
\begin{equation}
    E_1(X_r,n):=\frac{H^*(X^n)[g_{ij},g_i^s:i\neq j\in [n], 1\leq s\leq r]}{(\mathrm{relations})}
\end{equation}
with $H^p(X^n)$ having bidegree $(p,0)$, and $g_{ij}$ and $g_i^s$ having bidegree $(2d,1)$ and Hodge type $(d,d)$, subject to relations
\begin{align}
g_{ij}&=g_{ji},\label{relation1}\\
g_{ij}g_{jk}+g_{jk}g_{ki}+g_{ki}g_{ij}&=0 \text{ for }i,j,k\text{ distinct},\label{relation2}\\
g_{ij}\alpha_i&=g_{ij}\alpha_j\text{ for }\alpha\in H^*(X),\label{relation3}\\
g_i^s \alpha_i&=0\text{ for }\alpha\in H^{\geq 1}(X),\label{relation4}\\
g_i^s g_j^s - g_{ij}g_i^s + g_{ij}g_j^s&=0,\label{relation5}\\
g_i^s g_i^t &= 0\text{ for }s\neq t.\label{relation6}
\end{align}

The differential $\dee$ on $E_1(X_r,n)$ is determined by 
\begin{align}
\dee g_{ij}&=p_{ij}^*[\Delta],\\
\dee g_{i}^s&=[P^s]_i,\\
\dee|_{H^*(X^n)}&=0
\end{align}

The symmetric group $S_n$ acts on $E_1(X_r,n)$ by permuting subscripts.
\end{proposition}

Again, the technical point that \eqref{relation1}--\eqref{relation6} generate all the relations in Theorem~\ref{thm:totaro} will be addressed in the Appendix. 

\begin{remark}
    Note that the ordered configuration space $F(X_r,n)$ can be treated in at least two ways: the complement of the braid-puncture arrangement on $X$, and the complement of the braid arrangement of $X_r$. A key point of our method is to use the spectral sequence from the former interpretation, instead of the latter. It will turn out that we gain two major advantages at the price of a little more combinatorics: (i) it is easier to compare different $r$'s, and (ii) we can now take advantage of mixed Hodge purity of $X$.
\end{remark}

\section{An explicit isomorphism}\label{sec:isom}
In this Section, we will construct an explicit isomorphism \eqref{eq:key-isom} that is compatible with the $(i,j)$ bigrading, the Hodge type bigrading, and the $S_n$ action. The observation of this isomorphism is our first new ingredient. Our isomorphism will relate $E_1(X_r,n)$ to $E_1(X_{r-1},n)$ and $E_1(X_r,n-1)$. We obtain the isomorphism by manipulating the presentation of $E_1(X_r,n)$ in Proposition~\ref{prop:new}.

\subsection{Technical notes}
As usual, for a graded-commutative ring $R$, $R\pairing{x_1,\dots,x_m}$ denotes the free left $R$-module with basis $x_1,\dots,x_m$, and if $x_1,\dots,x_m$ are understood as indeterminates homogeneous of certain degrees, $R[x_1,\dots,x_m]$ denotes the free graded-commutative algebra over $R$ generated by $x_1,\dots,x_m$. More precisely, $R[x_1,\dots,x_m]$ is a free left $R$-module with basis $x_1^{n_1}\cdots x_m^{n_m}$, $n_i\geq 0$, and the multiplication rule is given by the graded commutativity according to the parities of $\deg(x_i)$. 

We will repetitively use the following fact. If for all $i<j\in [m]$, and for all $i=j\in [m]$ such that $\deg(x_i)$ is even, we have a left $R$-linear combination $L_{ij}=L_{ij}(x_1,\dots,x_m)$ of $x_1,\dots,x_m$, then as left $R$-modules:
\begin{equation}\label{eq:quadratic}
    \frac{R[x_1,\dots,x_m]}{(x_ix_j-L_{ij})_{i,j\in [m]}} \cong R\pairing{1,x_1,\dots,x_m},
\end{equation} 
and the multiplication rule on $R\pairing{1,x_1,\dots,x_m}$ can be read from $x_ix_j=L_{ij}$. 

The following instance of \eqref{eq:quadratic} will be relevant. Let $A$ be a basis of $H^{\geq 1}(X)$ consisting of homogeneous elements. Since $H^*(X)=\Q \oplus H^{\geq 1}(X)$ (as $X$ is connected), we have $H^*(X)=\Q\langle 1, \alpha: \alpha\in A\rangle$. Let $L^\gamma_{\alpha\beta}\in \Q$ be the structure constants for the multiplication in $H^*(X)$:
\begin{equation}\label{eq:cup-struct-const}
    \alpha\beta = \sum_{\gamma\in A} L^\gamma_{\alpha\beta} \gamma.
\end{equation}
Then $H^*(X)=\Q[\alpha:\alpha\in A]/\eqref{eq:cup-struct-const}$ as algebras. By K\"unneth's formula, $H^*(X^n)=\Q[\alpha_i:i\in [n], \alpha\in A]/(\text{relations})$ with relations
\begin{equation}
    \alpha_i\beta_i = \sum_{\gamma\in A} L^\gamma_{\alpha\beta} \gamma_i.
\end{equation}

For any finite set $I$, we treat $E_1(X_r,I)$ as an isomorphic copy of $E_1(X_r,\abs{I})$ but with subscripts $i,j$ of the generators $g_{ij},g_i^s$ taken from $I$. 

\subsection{Main derivation}
\subsubsection*{Step 1} We first present $E_1(X_r,n)$ over $E_1(X_{r-1},n)$. 

\

Applying Proposition~\ref{prop:new} to both $r$ and $r-1$, and separating the old and new generators and relations, we get
\begin{equation}
E_1(X_r,n)=E_1(X_{r-1},n)[g_1^r,\dots,g_n^r]/(\text{new relations}),
\end{equation}
with new relations
\begin{align}
g_i^r g_j^r &= g_{ij} g_i^r - g_{ij} g_j^r \text{ for }i\neq j\in [n], \label{newrelation1}\\
g_i^s g_i^r &= 0\text{ for }1\leq s\leq r-1, \label{newrelation2}\\
g_i^r \alpha_i &= 0\text{ for }\alpha\in H^{\geq 1}(X). \label{newrelation3}
\end{align}

By \eqref{eq:quadratic}, 
\begin{equation}
    \frac{E_1(X_{r-1},n)[g_i^r: i\in [n]]}{\eqref{newrelation1}} \cong E_1(X_{r-1},n)\langle 1,g_i^r: i\in [n]\rangle
\end{equation}
as left $E_1(X_{r-1},n)$-modules, with multiplication determined by \eqref{newrelation1}. Noting that \eqref{newrelation2} and \eqref{newrelation3} simply say that some basis vectors $g_i^r$ are annihilated by ``scalars'' $g_i^s, \alpha_i\in E_1(X_{r-1},n)$, we obtain
\begin{align}
    E_1(X_r,n)&=\frac{E_1(X_{r-1},n)[g_i^r: i\in [n]]/(\ref{newrelation1})}{(\eqref{newrelation2},\eqref{newrelation3})}\\
    &=\frac{E_1(X_{r-1},n)\langle 1,g_i^r: i\in [n]\rangle}{(g_i^s g_i^r, \alpha_i g_i^r: i\in [n], s\neq r,\alpha\in H^{\geq 1}(X))} \label{eq:presentationE1Xstar}\\
    &=E_1(X_{r-1},n)\oplus \bigoplus_{i=1}^n \frac{E_1(X_{r-1},n)} {(g_i^s, \alpha_i: s\neq r)}\langle g_i^r\rangle.\label{eq:decompositionE1Xstar}
\end{align}

\subsubsection*{Step 2} For each fixed $i\in [n]$, we present $E_1(X_{r-1},n)$ over $E_1(X_{r-1},[n]-i)$.

\

Comparing the presentations of $E_1(X_{r-1},n)$ and $E_1(X_{r-1},[n]-i)$ in Proposition \ref{prop:new}, and recalling \eqref{eq:cup-struct-const}, we get the following presentation, where $i$ is fixed, the subscripts $j,k$ range over $[n]-i$, the superscripts $s,t$ range over $[r-1]$, and $\alpha$ is taken from a basis $A$ of $H^{\geq 1}(X)$:
\begin{equation}\label{eq:presentationE1X}
E_1(X_{r-1},n)=\frac{E_1(X_{r-1},[n]-i)[g_{ij}, g_i^s,\alpha_i]}{\left(
\begin{array}{rcll}
g_{ij}g_{ik} &=& -g_{jk}g_{ij}+g_{jk}g_{ik}, &j\neq k\\
g_{ij}g_i^s &=& -g_j^s g_i^s - g_j^s g_{ij}, &j,s\\
g_i^s g_i^t &=& 0, &s\neq t\\
\alpha_i g_{ij} &=& \alpha_j g_{ij}, &j, \alpha\\
g_i^s \alpha_i &=& 0, &s, \alpha\\
\alpha_i \beta_i &=& \sum_{\gamma\in A} L^\gamma_{\alpha\beta} \gamma_i, &\alpha,\beta
\end{array}
\right)
}.
\end{equation}

As before, we will apply \eqref{eq:quadratic}, and get
\begin{equation}\label{eq:decompositionE1X}
    E_1(X_{r-1},n)=E_1(X_{r-1},[n]-i)\langle 1,g_{ij},g_i^s,\alpha_i\rangle,
\end{equation}
with multiplication given by \eqref{eq:presentationE1X}. 

\subsubsection*{Step 3}
We substitute \eqref{eq:decompositionE1X} into the $i$-th summand in \eqref{eq:decompositionE1Xstar}.

\

In the following computation, $i$ is fixed, while $j\in [n]-i$, $s,t\in [r-1]$, $\alpha,\beta \in H^{\geq 1}(X)$ vary. Note that the multiplication structure of $E_1(X_{r-1},n)$ plays a role in \eqref{eq:decompositionE1Xstar}, so when we substitute $E_1(X_{r-1},n)$ using \eqref{eq:decompositionE1X}, we must extract the multiplication from \eqref{eq:presentationE1X}.

\begin{align}
    \frac{E_1(X_{r-1},n)} {(g_i^s, \alpha_i)}
    &=\frac{E_1(X_{r-1},[n]-i)\langle 1,g_{ij},g_i^s,\alpha_i\rangle}{E_1(X_{r-1},[n]-i)(1,g_{ij},g_i^t,\beta_i) (g_i^s, \alpha_i)}\\
    &=\frac{E_1(X_{r-1},[n]-i)\langle 1,g_{ij},g_i^s,\alpha_i\rangle}{\left(
        \begin{array}{rrrr}
        g_i^s \beta_i, &g_i^s, &g_i^s g_{ij}, &g_i^s g_i^t\\
        \alpha_i \beta_i, &\alpha_i, &\alpha_i g_{ij}, &\alpha_i g_i^t
        \end{array}
        \right)} \\
    &\overset{\eqref{eq:presentationE1X}}{=}\frac{E_1(X_{r-1},[n]-i)\langle 1,g_{ij},g_i^s,\alpha_i\rangle}{\left(
        \begin{array}{rrrr}
        0, &g_i^s, &g_j^s g_i^s+g_j^s g_{ij}, &0\\
        \sum_\gamma L^\gamma_{\alpha\beta}\gamma_i, &\alpha_i, &\alpha_j g_{ij}, &0
        \end{array}
        \right)}.
\end{align}

Two of the relations remove the generators $g_{i}^s$ and $\alpha_i$, so we are left with
\begin{align}
    \frac{E_1(X_{r-1},n)} {(g_i^s, \alpha_i)} &= \frac{E_1(X_{r-1},[n]-i)\langle 1,g_{ij}\rangle}{(g_j^s g_{ij}, \alpha_j g_{ij})} \\
    &= E_1(X_{r-1},[n]-i) \oplus \bigoplus_{j\in [n]-i} \frac{E_1(X_{r-1},[n]-i)}{(g_j^s,\alpha_j)} \langle g_{ij} \rangle. \label{eq:presentation_summand}
\end{align}

\subsubsection*{Step 4} We observe a coincidence based on our computation.

\

We surprisingly note that \eqref{eq:presentation_summand} is isomorphic to \eqref{eq:decompositionE1Xstar} applied to the index set $[n]-i$, up to renaming the basis vectors in the summations. Namely, for each $i\in [n]$, we have an isomorphism
\begin{equation}
    E_1(X_r,[n]-i) \to \frac{E_1(X_{r-1},n)} {(g_i^s, \alpha_i)}
\end{equation} 
that sends $1$ to $1$, $g_j^r$ to $g_{ij}$, and preserves $E_1(X_{r-1},[n]-i)$. 

Applying this to each summand of the \emph{original} \eqref{eq:decompositionE1Xstar}, we obtain an isomorphism summarized in the lemma below.

\begin{lemma}\label{lem:decomposition}
    There exists a $\Q$-linear isomorphism
    \begin{equation}\label{eq:key-isom}
        \Phi: E_1(X_{r-1},n)\oplus \bigoplus_{i=1}^n E_1(X_{r},[n]-i) \to E_1(X_r,n),
    \end{equation}
    such that $\Phi|_{E_1(X_{r-1},n)}$ is the natural map to $E_1(X_r,n)$, and $\Phi|_{E_1(X_r,[n]-i)}$ is the $E_1(X_{r-1},[n]-i)$-linear map that sends $1$ to $g_i^r$ and sends $g_j^r$ to $g_{ij}g_i^r$ for all $j\in [n]-i$. 
\end{lemma}

Note that $\Phi$ is uniquely determined by the properties stated in Lemma \ref{lem:decomposition}, since $1$ and $g_j^r$, $j\in [n]-i$ generate $E_1(X_{r},[n]-i)$ as an $E_1(X_{r-1},[n]-i)$-module. The content of Lemma \ref{lem:decomposition} is that $\Phi$ is well-defined and is an isomorphism.

\subsection{Compatibility with differential}
Lemma \ref{lem:decomposition} would be useless if $\Phi$ were not compatible with the differential $\mathrm{d}=d_1$. However, we have the following.

\begin{lemma}\label{lem:compat-diff}
    Suppose $X$ is non-compact, then $\Phi$ in \eqref{eq:key-isom} commutes with $\mathrm{d}$.
\end{lemma}
\begin{proof}
    Recall the differential rule from Proposition~\ref{prop:new}: $\mathrm{d}g_{jk}=p_{jk}^*[\Delta]$, $\mathrm{d}g_{j}^s=[P^s]_i$, and $\mathrm{d}\alpha_j=0$ for all $j,k\in [n], s\in [r]$. Moreover, if $X$ is non-compact, then we recall from the discussion after \eqref{eq:fund-class} that $[P^s]=0$, so 
    \begin{equation}\label{eq:diff-vanish}
        \mathrm{d}g_{j}^s=0,\; j\in [n],\; s\in [r]
    \end{equation}

    To verify that $\Phi$ commutes with $\mathrm{d}$, it suffices to check that $\mathrm{d}1=\mathrm{d}g_i^r$ and $\mathrm{d}g_j^r=\mathrm{d}(g_{ij}g_i^r)$. By \eqref{eq:diff-vanish}, it only remains to check that $\mathrm{d}(g_{ij}g_i^r)=0$. Consider
    \begin{equation}
        \mathrm{d}(g_{ij}g_i^r)=\mathrm{d}(g_{ij}) g_i^r-g_{ij}\mathrm{d}g_i^r=\mathrm{d}(g_{ij})g_i^r=p_{ij}^*[\Delta] g_i^r.
    \end{equation}

    By K\"unneth's formula, $[\Delta]\in H^{2d}(X^2)=\bigoplus_{p=0}^{2d} H^p(X)\otimes H^{2d-p}(X)$. Since $X$ is non-compact, $H^{2d}(X)=0$, so $\bigoplus_{p=1}^{2d-1}$ suffices. In particular, $p_{ij}^*[\Delta]$ can be expressed as a $\Q$-linear combination of terms of the form $\beta_j\alpha_i$, where $\alpha,\beta\in H^{\geq 1}(X)$. 

    But relation \eqref{relation4} implies $\alpha_i g_i^r=0$ for all $\alpha\in H^{\geq 1}(X)$. This implies $\mathrm{d}(g_{ij}g_i^r)=0$, as required.
\end{proof}

We remark that Lemma \ref{lem:compat-diff} is the only reason why we need $X$ to be non-compact.

\section{Proof of our results}\label{sec:proof}

In this Section, we finish the proofs of our results.

\subsection{Proof of Theorem \ref{thm:C}}
\begin{proof}
    Let $X$ be a non-compact pure variety of dimension $d$ such that $Y=X_{r-1}=X-P^1-\dots-P^{r-1}$. Let $P=P^s$, so that $Y-P=X_r$. By our spectral sequences, we have a non-canonical isomorphism of $S_n$-representations
    \begin{equation}
        H^k(F(X_r,n))\cong \bigoplus_{i-j=k} E_\infty^{i,j}(X_r,n). 
    \end{equation} 

    Consider $\Phi$ in \eqref{eq:key-isom}. By Lemma \ref{lem:compat-diff}, $\Phi$ descends to an isomorphism
    \begin{equation}
        E_2(X_{r-1},n)\oplus \bigoplus_{i=1}^n E_2(X_{r},[n]-i) \to E_2(X_r,n).
    \end{equation}
   The direct sum can be recognized as $\Ind_{S_{n-1}}^{S_n} E_2(X_r,n-1)$. By Lemmas \ref{lem:degen} and \ref{lem:good}, $E_2=E_\infty$. Thus we get
    \begin{equation}
        E_\infty(X_{r-1},n)\oplus \Ind_{S_{n-1}}^{S_n} E_\infty(X_{r},n-1) \cong E_\infty(X_r,n).
    \end{equation}
    
    Keeping track of how $\Phi$ is defined on the generators (see Lemma \ref{lem:decomposition}) and the fact that $g_{ij}, g_i^s$ have bidegree $(2d,1)$ and Hodge type $(d,d)$, we have an $S_n$-equivariant isomorphism for each $i,j,p,q$:
    \begin{equation}
        H^{p,q}(E_\infty^{i,j}(X_{r-1},n))\oplus \Ind_{S_{n-1}}^{S_n} H^{p-d,q-d}(E_\infty^{i-2d,j-1}(X_r,n-1)) \to H^{p,q}(E_\infty^{i,j}(X_r,n)).
    \end{equation}

    For each $p,q,k$, by summing over all $i,j$ with $i-j=k$, we get
    \begin{equation}
        H^{p,q;k}(F(X_{r-1},n)) \oplus \Ind_{S_{n-1}}^{S_n} H^{p-d,q-d;k-(2d-1)}(F(X_r,n-1)) \cong H^{p,q;k}(F(X_r,n)).
    \end{equation}
    This can be equivalently restated as
    \begin{equation}\label{eq:final}
        \gr H^{k}(F(X_{r-1},n)) \oplus \Ind_{S_{n-1}}^{S_n} \gr H^{k-(2d-1)}(F(X_r,n-1))(-d) \cong \gr H^{k}(F(X_r,n)).
    \end{equation}
    For the term with $F(X_r,n-1)$, we may decompose it using \eqref{eq:final} again, getting a term with $F(X_{r-1},n-1)$ and a term with $F(X_r,n-2)$. Continuing the process with induction, \eqref{eq:equivariant_mhs} follows as desired.
\end{proof}

\subsection{Proof of Theorem \ref{thm:A}}

\begin{proof}
    It suffices to take the $S_n$-invariants of \eqref{eq:equivariant_mhs} and note by Frobenius reciprocity that $(\Ind_{S_{n-k}}^{S_n} M)^{S_n}=M^{S_{n-k}}$ for any $S_{n-k}$-representation $M$. 
\end{proof}

\subsection{Proof of Corollary \ref{thm:D}}

\begin{proof}
    We first work out how character polynomials and Frobenius characteristics interact with induction. Let $1_k$ denote the trivial subgroup of $S_k$. Suppose $M=\{M_n\}_{n\geq 0}$ is a sequence of $S_n$-representations. Let $I_k M$ be the sequence $\{(I_k M)_n\}_{n\geq 0}$ such that
    \begin{equation}
        (I_k M)_n:=\begin{cases}
            \Ind_{S_{n-k}}^{S_n}M_{n-k}:=\Ind_{S_{n-k}\times 1_k}^{S_n} (M_{n-k}\boxtimes \Q),& n\geq k; \\
            0,& 0\leq n<k.
        \end{cases}
    \end{equation}
    Clearly $I_{k+l}M\cong I_k(I_l M)$. 

    Suppose $M$ admits a character polynomial $P(x_1,x_2,\dots)\in \Q[x_1,x_2,\dots]$, namely, 
    \begin{equation}
        \chi_{M_n}(\sigma)=P(x_1(\sigma),\dots) \text{ for each $n\geq 0$ and $\sigma\in S_n$,}
    \end{equation}
    see \eqref{eq:charpoly}. We now claim $I_k M$ admits a character polynomial
    \begin{equation}\label{eq:induct-char-poly}
        I_k P(x_1,\dots):=x_1(x_1-1)\cdots (x_1-k+1) P(x_1-k,x_2,x_3,\dots).
    \end{equation}
    
    Since $I_k$ is just $I_1$ applied $k$ times, it suffices to prove this for $k=1$. Recall that for any class function $f$ on $S_{n-1}$, the induced class function is given by
    \begin{equation}
        \left(\Ind_{S_{n-1}\times 1}^{S_n}f\right)(\sigma) = \sum_{i:\sigma(i)=i} f(\sigma|_{[n]-i}).
    \end{equation}
    If $\sigma(i)=i$, $\sigma|_{[n]-i}$ has all cycles of $\sigma$ except for a $1$-cycle (the fixed point $i$). Thus
    \begin{align}
        \chi_{(I_1 M)_n}(\sigma)&=\parens*{\Ind_{S_{n-1}\times 1}^{S_n} \chi_{M_{n-1}}}(\sigma)\\
        &= \sum_{i:\sigma(i)=i} \chi_{M_{n-1}}(\sigma|_{[n]-i}) \\
        &= \sum_{i:\sigma(i)=i} P(x_1-1,x_2,x_3,\dots)(\sigma)\\
        &= (x_1 P(x_1-1,x_2,\dots))(\sigma)
    \end{align}
    for all $n\geq 1$ and $\sigma\in S_n$. If $n=0$, both sides are 0. This concludes \eqref{eq:induct-char-poly}. 

    Similarly, let $\lambda$ be a partition and $M=M(V_\lambda)$ as in \cite[p.~1848]{cef2015fi}. By [ibid, (4)], 
    \begin{equation}
        M_n=\Ind_{S_{n-\abs{\lambda}}\times S_{\abs{\lambda}}}^{S_n} (\Q \boxtimes V_{\lambda}),
    \end{equation}
    where $\Q$ denotes the trivial representation. It follows that
    \begin{align}
        (I_k M(V_\lambda))_n &= \Ind_{S_{n-k}\times 1_k}^{S_n} M(V_\lambda)_{n-k} \\
        &= \Ind_{S_{n-k}\times 1_k}^{S_n} \Ind_{S_{n-k-\abs{\lambda}}\times S_{\abs{\lambda}}\times 1_k}^{S_{n-k}\times 1_k} (\Q \boxtimes V_{\lambda} \boxtimes \Q) \\
        &=\Ind_{S_{n-k-\abs{\lambda}}\times S_{\abs{\lambda}}\times 1_k}^{S_n} (\Q \boxtimes V_{\lambda} \boxtimes \Q)\\
        &=\Ind_{S_{n-(k+\abs{\lambda})}\times S_{\abs{\lambda}+k}}^{S_n} \left( \Q \boxtimes \Ind_{S_{\abs{\lambda}}\times 1_k}^{S_{\abs{\lambda}+k}}(V_\lambda \boxtimes \Q) \right)\\
        &=(M(I_k V_{\lambda}))_n,
    \end{align}
    where $I_k V_\lambda$ denotes the $S_{\abs{\lambda}+k}$-representation $\Ind_{S_{\abs{\lambda}}\times 1_k}^{S_{\abs{\lambda}+k}}(V_\lambda \boxtimes \Q)$. Hence
    \begin{equation}
        I_kM(V_\lambda)=M(I_k V_\lambda).
    \end{equation}

    A standard exercise in the induction ring implies that the Frobenius characteristics of $I_k V_\lambda$ is $s_1^k s_\lambda$. If for any $M$ of the form $M=\oplus_{\lambda \in \mathcal P} c_\lambda M(V_\lambda)$, we define $\mathrm{ch}(M):=\sum c_\lambda s_\lambda$ (cf.~\eqref{eq:fisharp}), then the discussion above implies
    \begin{equation}\label{eq:induct-frob-char}
        \mathrm{ch}(I_kM) = s_1^k \mathrm{ch} (M).
    \end{equation}

    We are now ready to prove both assertions of Corollary \ref{thm:D}. From \eqref{eq:equivariant}, we have
    \begin{equation}
        H^i(F(Y-P,-))=\bigoplus_{k\geq 0} I_k H^{i-k(2d-1)}(F(Y,-)).
    \end{equation}
    Then \eqref{eq:charpoly} follows immediately from \eqref{eq:induct-char-poly}, and \eqref{eq:frobchar} follows from \eqref{eq:induct-frob-char}.
\end{proof}

\subsection{Proof of Corollary \ref{cor:generating}}

\begin{proof}
    Theorem \ref{thm:A} implies the identity on the generating function \eqref{eq:hx-def}:
    \begin{equation}
        h_{Y-P}(x,y,u;t):=\frac{1}{1+(xy)^d u^{2d-1}t} h_Y(x,y,u;t).
    \end{equation}
    We may repeat the process, since if $Y$ is a possibly punctured non-compact pure variety then so is $Y-P$. It follows that
    \begin{equation}\label{eq:more-punc}
        h_{Y_s}(x,y,u;t):=\frac{1}{(1+(xy)^d u^{2d-1}t)^s} h_Y(x,y,u;t),
    \end{equation}
    where $Y_s$ is $Y$ minus $s$ points. Let $Y=\Sigma_{g,1}$, so that $Y_{r-1}=\Sigma_{g,r}$. For a given $g$ and $r\geq 1$, by \eqref{eq:more-punc} with $d=1,s=r-1$, the truth of \eqref{eq:conj:generating} for $\Sigma_{g,1}$ is equivalent to the case for $\Sigma_{g,r}$.

    One may check that \eqref{eq:conj:generating} for $g=r=1$ is just a reformulation of \cite[Thm.~1.1]{cheonghuang2022betti}. This implies \eqref{eq:conj:generating} for $g=1, r\geq 1$ as claimed. 
\end{proof}

\section{Further directions}
Our theorems naturally motivate the following questions.

\begin{question}
    In Theorem \ref{thm:A} and/or Theorem \ref{thm:C}, how much can we relax the assumption that $Y$ is a possibly punctured non-compact pure variety? For example, do they hold for every non-compact variety $Y$? 
\end{question}

It is necessary that $Y$ is non-compact; recall that even \eqref{eq:kallel} does not hold for an elliptic curve $\Sigma_{1,0}$ by \cite[Tables~2~and~3]{napolitano2003cohomology}. From the perspective of our method, the need of non-compactness of $Y$ is manifest in the proof of Lemma \ref{lem:compat-diff}. On the other hand, we remark that the construction of our key isomorphism $\Phi$ does not require any mixed-Hodge-theoretic assumption on $Y$ (not even a complex structure on $Y$). Any other setting where the degeneration Lemma \ref{lem:degen} holds would produce a generalization of \eqref{eq:kallel}, Theorem \ref{thm:A}, and Theorem \ref{thm:C}.

\section{Appendix: the braid-puncture arrangement}

In this Appendix, we provide some standard combinatorial details about the braid-puncture arrangement, and prove that \S \ref{sec:spectral-new} has listed all relations.

Let $\cA$ be the braid-puncture arrangement in \S \ref{sec:spectral-new}. Strata of $\cA$ correspond bijectively to pairs $(\chi,\sim)$ of a \emph{coloring} function $\chi:\{1,\dots,n\}\to \{0,\dots,r\}$ and an equivalence relation $\sim$ on $\chi^{-1}(0)$, according to the rule
\begin{equation}
    \begin{aligned}
        F_{(\chi,\sim)} &= \{(x_1,\dots,x_n)\in X^n: \text{$x_i=P^{\chi(i)}$ if $\chi(i)\neq 0$,} \\
        &\text{and $x_i=x_j$ if $\chi(i)=\chi(j)=0$ and $i\sim j$. }\}
    \end{aligned}
\end{equation}

In other words, a coordinate that is colored $s$ ($1\leq s\leq r$) is required to take $P^s$ as value, and the coordinates colored $0$ (``uncolored'') have no such a requirement, but they must agree if related by $\sim$. 
 
The following describe the map $H^*(X^n)\to H^*(F)$. Recall the projections $p_i:X^n\to X$.

\begin{lemma}\label{lem:surjective}
Let $F=F_{(\chi,\sim)}$, then $H^*(X^n)\to H^*(F)$ is surjective with kernel generated by $p_i^*\alpha$ for $\alpha\in H^{\geq 1}(X), \chi(i)\neq 0$, and $p_i^*\alpha - p_j^*\alpha$ for $\alpha\in H^*(X), \chi(i)=\chi(j)=0, i\sim j$.  
\end{lemma}
\begin{proof}
We can identify $F$ with the space of functions from the quotient set $\chi^{-1}(0)/{\sim}$. The corresponding inclusion map $\iota: X^{\chi^{-1}(0)/{\sim}}\incl X^n$ sends a function $f$ to the function $g:[n]\to X$ defined by (here $\bar i$ means the equivalence class of $i$)
\begin{equation}\label{eq:stratum-incl}
    g(i)=\begin{cases}
        P^{\chi(i)}, &\text{ if }\chi(i)\neq 0; \\
        f(\bar i), &\text{ if }\chi(i)=0.
    \end{cases}
\end{equation}

Let $R=H^*(X)$, and for $i\in [n]$, denote for convenience $\theta_i=p_i^*: H^*(X)\to H^*(X^n)$. K\"unneth's formula says $H^*(X^n)=R^{\otimes n}$, in which $\theta_1(a_1)\cdots \theta_n(a_n)$ is identified with $a_1\otimes \dots \otimes a_n$ for $a_i\in R$. Let $\epsilon:R\to \Q$ be quotient map by $H^{\geq 1}(X)$ (note that $X$ is connected so $H^0(X)=\Q$), which is also the pullback of the constant map $X\to \mathrm{pt}$. 

Let $\mu: R^{\otimes n} \to R^{\otimes \chi^{-1}(0)/{\sim}}$ be the pullback by $\iota$. We now describe $\mu(\theta_i(a))$ for $i\in [n]$ and $a\in R$. This is the pullback by $p_i\circ \iota$. It follows from that \eqref{eq:stratum-incl} that
\begin{equation}\label{eq:mu}
    \mu(\theta_i(a))=\begin{cases}
        \epsilon(a), &\text{ if }\chi(i)\neq 0; \\
        \theta_{\bar i}(a), &\text{ if }\chi(i)=0.
    \end{cases}
\end{equation} 

Since $\theta_{\bar i}(a)$ for $\bar i\in \chi^{-1}(0)/{\sim}$ and $a\in R$ generate $R^{\otimes \chi^{-1}(0)/{\sim}}$, we know $\mu$ is surjective. Let $J$ be the ideal of $R^{\otimes n}$ generated by $\theta_i(H^{\geq 1}(X))$ for $i\notin \chi^{-1}(0)$ and $\theta_i(a)-\theta_j(a)$ for $i,j\in \chi^{-1}(0), i\sim j$. It suffices to prove $J=\ker(\mu)$. 

Clearly $J\subeq \ker(\mu)$ by \eqref{eq:mu}. A useful trick to show the equality is constructing an inverse to the natural map $R^{\otimes n}/J\to R^{\otimes n}/\ker(\mu)\cong R^{\otimes \chi^{-1}(0)/{\sim}}$. To construct the desired map $\nu: R^{\otimes \chi^{-1}(0)/{\sim}}\to R^{\otimes n}/J$, by the universal property of tensor products, it suffices to construct $\nu_{\bar i}:R\to R^{\otimes n}/J$. We simply define
\begin{equation}
    \nu_{\bar i} (a) = \theta_i(a)\bmod J,
\end{equation}
where $i\in \chi^{-1}(0)$ is any lift of $\bar i$. This is well-defined because if $i,j$ are two lifts, then $\theta_i(a)-\theta_j(a)\in J$. To check that $\nu$ is the desired inverse, the only place that requires remarks is verifying that $\nu(\mu(\theta_i(a)))-\theta_i(a) \in J$ if $\chi(i)\neq 0$. Indeed, since $\mu(\theta_i(a))=\epsilon(a)\in \Q$, and $\theta_i$ is identity on $\Q$, we have
\begin{equation}
    \nu(\mu(\theta_i(a)))-\theta_i(a) = \epsilon(a)-\theta_i(a)=\theta_i(\epsilon(a)-a),
\end{equation}
which is in $J$ since $\epsilon(a)-a\in H^{\geq 1}(X)$.
\end{proof}

\begin{remark}
    It follows that \eqref{relation3} and \eqref{relation4} generate all relations from \eqref{eq:rel-totaro}.
\end{remark}

Next, we deal with the Orlik--Solomon relations. Let $B(\cA)=\Q[g_Y:Y\in \cA]$ as a free graded-commutative algebra with degree-one elements $e_Y$, with a derivation $\partial$ given by $\partial g_Y=1$. Let $I(\cA)$ be the ideal generated by the Orlik--Solomon relations \eqref{eq:os-rel}. Then $A(\cA)=B(\cA)/I(\cA)$.

\begin{lemma}\label{lem:os-braid}
    For the braid-puncture arrangement $\cA$ as before, let $J(\cA)$ be the ideal of $B(\cA)$ generated by \eqref{relation2}, \eqref{relation5}, and \eqref{relation6}. Then $J(\cA)=I(\cA)$.
\end{lemma}

\begin{remark}
    From now on, we work in $B(\cA)$ rather than $A(\cA)$, which means $g_Y$ no longer satisfies the Orlik--Solomon relations. 
\end{remark}

We will need some lemmas.

\begin{lemma}[{\cite[Lem.~3.7]{orlikterao1992}}]
    For $\varnothing\neq S\subeq \cA$, we have $g_S\in \partial g_S B(\cA)$. 
\end{lemma}
\begin{proof}
    Pick $Y\in S$. Then $g_Y g_S=0$ by graded commutativity, so $0=\partial (g_Y g_S)=g_S-g_Y \partial g_S$. 
\end{proof}

\begin{remark}\label{rmk:minimal}
    This implies the Orlik--Solomon relations are generated by $g_S$ for minimal vanishing sets $S$ and $\partial g_S$ for minimal dependent sets $S$. The key is that if $\varnothing \neq S\subeq T$, then by Leibniz rule, $\partial g_T\in (g_S,\partial g_S)B(\cA)=\partial g_S B(\cA)$ by the lemma.
\end{remark}

\begin{lemma}\label{lem:chainreduction}
    If $S_1\subeq \cA$, $S_2\subeq \cA$, and $Y\in \cA$ are all disjoint, then 
    \begin{equation}
        \partial g_{S_1\sqcup S_2} \in (\partial g_{S_1\sqcup \{Y\}}, \partial g_{S_2\sqcup \{Y\}}) B(\cA).
    \end{equation}
\end{lemma}

\begin{proof}
Write $A=g_{S_1}$, $B=g_{S_2}$ and $g=g_Y$. We have 
\begin{align}
\partial(g\partial(AB)) &= (\partial g)(\partial(AB))-g\partial^2(AB) \\
&=1\cdot \partial(AB)-0\\
&=\partial(AB),\\
\partial(AB)&= \partial(e\partial(AB))\\
&=\pm \partial(gB\partial A) \pm \partial(gA\partial B)\\
&=\pm \partial(gB)\partial(A) \pm \partial(gA)\partial B \text{ (since $\partial^2=0$)}\\
&\in (\partial(gA),\partial(gB)) B(\cA).\qedhere
\end{align}
\end{proof}

\begin{proof}[{Proof of Lemma \ref{lem:os-braid}}]
    Clearly $J(\cA)\subeq I(\cA)$. By Remark \ref{rmk:minimal}, it suffices to show that $g_S\in J(\cA)$ for minimal vanishing sets $S$ and $\partial g_S\in J(\cA)$ for minimal dependent sets $S$. These elements are (1) $g_i^s g_j^t g_\gamma$, $s\neq t$; (2) $\partial(g_\lambda)$; and (3) $\partial(g_i^s g_j^s g_\gamma)$,
    where 
    \begin{itemize}
        \item In (1)(3), $\gamma=(i=i_0\to i_1\to\dots\to i_h=j)$ is a path joining $i$ and $j$, and $g_\gamma:=g_{i_0 i_1}g_{i_1 i_2}\dots g_{i_{h-1} i_h}$. Here $h\geq 1$ and $i_0,\dots,i_h$ are distinct. 
        \item In (2), $\lambda=(i_1\to \dots \to i_h\to i_1)$ is a loop, and $g_\lambda:=g_{i_1 i_2} g_{i_2 i_3} \dots g_{i_{h-1} i_h} g_{i_h i_1}$. Here $h\geq 3$ and $i_1,\dots,i_h$ are distinct.
    \end{itemize}

    First, we prove that (2) is in $J(\cA)$ by induction on $h$. If $h=3$, then (2) is just \eqref{relation2}. If $h>3$, we set $S_1=\{\Delta_{i_1 i_2},\Delta_{i_2 i_3},\dots,\Delta_{i_{h-2} i_{h-1}}\}, S_2=\{\Delta_{i_{h-1} i_h}, \Delta_{i_h i_1}\}$, and $Y=\Delta_{i_1 i_{h-1}}$, then $\partial g_{S_1\sqcup Y}\in J(\cA)$ by the induction hypothesis, $\partial g_{S_2\sqcup Y}\in J(\cA)$ by the $h=3$ case. Applying Lemma \ref{lem:chainreduction}, we get $\partial(g_\lambda)=\partial g_{S_1\sqcup S_2}\in J(\cA)$.
    
    Next, we prove that (3) is in $J(\cA)$ by induction on $h$. The base case $h=1$ is \eqref{relation5}, and the induction step is proved similarly with $S_1=\{\Delta_i^s,\Delta_{i_0 i_1},\Delta_{i_1 i_2},\allowbreak \dots,\Delta_{i_{h-2} i_{h-1}}\}$, $S_2=\{\Delta_{i_{h-1} i_h}, \Delta_j^s\}$, and $Y=\Delta_{i_{h-1}}^s$.
    
    Finally, for (1), we note from the proven case (3) that
    \begin{equation}
        g_j^t g_\gamma - g_i^t \partial(g_i^t g_\gamma) = \partial(g_i^t g_j^t g_\gamma) \in J(\cA),
    \end{equation}
    so
    \begin{equation}
    g_j^t g_\gamma\equiv g_i^t \partial(g_i^t g_\gamma) \mod J(\cA).
    \end{equation}
    It follows that
    \begin{equation}
    g_i^s g_j^t g_\gamma \equiv g_i^s g_i^t \partial(g_i^t g_\gamma) \mod J(\cA).
    \end{equation}
    But $g_i^s g_i^t$ is just \eqref{relation6}, so $g_i^s g_j^t g_\gamma\in J(\cA)$.    
\end{proof}

\end{document}